\documentclass[11pt]{amsart}

\usepackage{amsmath,amssymb,latexsym,soul,cite,mathrsfs}

\usepackage{color,enumitem,graphicx}
\usepackage[colorlinks=true,urlcolor=blue,
citecolor=red,linkcolor=blue,linktocpage,pdfpagelabels,
bookmarksnumbered,bookmarksopen]{hyperref}
\usepackage[english]{babel}

\usepackage[left=2.9cm,right=2.9cm,top=2.8cm,bottom=2.8cm]{geometry}
\usepackage[hyperpageref]{backref}

\usepackage[colorinlistoftodos]{todonotes}
\makeatletter
\providecommand\@dotsep{5}
\def\listtodoname{List of Todos}
\def\listoftodos{\@starttoc{tdo}\listtodoname}
\makeatother

\numberwithin{equation}{section}

\newtheorem{theorem}{Theorem}[section]

\newtheorem{lemma}[theorem]{Lemma}

\newtheorem{remark}{Remark}

\begin{document}
	
	
	\title [Saddle solutions for Allen-Cahn type equations involving the prescribed mean curvature operator]{Saddle solutions for Allen-Cahn type equations involving the prescribed mean curvature operator}
	
	\author{Renan J. S. Isneri}
	
	\address[Renan J. S. Isneri]
	{\newline\indent Unidade Acad\^emica de Matem\'atica
		\newline\indent 
		Universidade Federal de Campina Grande,
		\newline\indent
		58429-970, Campina Grande - PB - Brazil}
	\email{\href{renan.isneri@academico.ufpb.br}{renan.isneri@academico.ufpb.br}}
	
	\pretolerance10000
	
	\begin{abstract}
		\noindent The goal of this paper is to investigate the existence of saddle solutions for some classes of elliptic partial differential equations of the Allen-Cahn type, formulated as follows:
		\begin{equation*}
			-div\left(\frac{\nabla u}{\sqrt{1+|\nabla u|^2}}\right) + A(x,y)V'(u)=0~~\text{ in }~~\mathbb{R}^2.
		\end{equation*}
		Here, the function $A:\mathbb{R}^2\to\mathbb{R}$ exhibits periodicity in all its arguments, while $V:\mathbb{R}\to\mathbb{R}$ characterizes a double-well symmetric potential with minima at $t=\pm\alpha$.
	\end{abstract}
	\thanks{}
	\subjclass[2020]{Primary: 34C37, 35J62, 35J93} 
	\keywords{Saddle solutions, Quasilinear elliptic equations, Mean curvature operator, Allen-Cahn equations}
	
	\maketitle


\section{Introduction}

Over the past 50 years, in view of several concrete real-world applications, considerable attention has been devoted to addressing the existence and classification of bounded solutions of the elliptic Allen-Cahn equations, whose prototype is given by
\begin{equation}\label{001}
	-\Delta u+V'(u)=0\quad \text{in}\quad \mathbb{R}^N,
\end{equation}
where the standard model of $V:\mathbb{R}\to\mathbb{R}$ is the classical double well Ginzburg-Landau potential
\begin{equation}\label{002}
	V(u)=\frac{1}{4}(u^2-1)^2.
\end{equation}
The concept of “{\it heteroclinic solutions}” has been extensively investigated in equations of the \eqref{001} type, leading to the emergence of numerous constructions documented in the literature. This type of solutions was the starting point to provide a rich amount of differently shaped families of solutions such as saddle, stationary layered and multibump solutions. Readers seeking a more comprehensive understanding of these solution types are encouraged to explore further in references such as \cite{Alessio0,Alessio1,Alessio2,AJM2,Rabinowitz1,RSlibro}.
 
 Equation \eqref{001} represents a reaction-diffusion type equation originating from mathematical physics. The initial model, encompassing \eqref{001}-\eqref{002}, was first proposed in 1979 by S. Allen and J. Cahn \cite{Allen} to delineate the phenomenon of phase separation within multi-component alloy systems. In this framework, the solution $u$ serves as a phase parameter, characterizing the material's state at each point, with the trivial solutions $u=-1$ and $u=+1$ corresponding to the pure phases of the system. Since then much progress has been made on  \eqref{001}-\eqref{002}, and moreover, its variations have been broadly treated in the PDE literature in the last years. These variations are mostly of one of the two types:
\begin{itemize}
	\item[(1)] The Laplacian operator is replaced by more general operators; \vspace{0.2cm}
	\item[(2)] \eqref{001} is changed to a class of more general semilinear elliptic equations whenever \eqref{002} is replaced by a class of potentials with different or even similar geometry.
\end{itemize}
Typically, both cases can occur simultaneously. We would like to cite some works involving generalizations and variations concerning problem \eqref{001}-\eqref{002}. In \cite{AAT} and \cite{Cabre2}, discussions revolve around results concerning the fractional Laplacian operator. On the other hand, \cite{Alves4} and \cite{Alves5} delve into outcomes associated with the $\Phi$-Laplacian operator, encompassing the classical $p$-Laplacian. Other interesting works involving the prescribed mean curvature operator are found in \cite{Alves1,Bonheure1}. Lastly, for a recent comprehensive overview of a fourth-order equation, we refer the reader to \cite{Bonheure}.

In this work, we will show that the following class of prescribed mean curvature equations
\begin{equation}\label{E1}
	-div\left(\frac{\nabla u}{\sqrt{1+|\nabla u|^2}}\right) + A(x,y)V'(u)=0~~\text{ in }~~\mathbb{R}^2,
\end{equation}
where $A(x,y)$ is $1$-periodic in all its variables, has a saddle solution whenever the distance between the roots of the double well symmetric potential $V$ is small. Moreover, in the particular case, when $A(x,y)$ is a positive constant, we will show that there is an infinite number of geometrically distinct saddle-type solutions for \eqref{E1}. Throughout this paper, the oscillatory factor $A(x,y)$ satisfies precisely the following assumptions:
\begin{itemize}
	\item[($A_1$)] $A$ is continuous and $A(x,y)>0$ for each $(x,y)\in\mathbb{R}^2$;
	\item[($A_2$)] $A(x,y)=A(-x,y)=A(x,-y)$ for all $(x,y)\in\mathbb{R}^2$;
	\item[($A_3$)] $A(x,y)=A(x+1,y)=A(x,y+1)$ for any $(x,y)\in\mathbb{R}^2$;
	\item[($A_4$)] $A(x,y)=A(y,x)$ for all $(x,y)\in\mathbb{R}^2$.
\end{itemize}
An interesting model for $A$ is described by the expression:
$$
A(x,y)=\cos(2\pi x)\cos(2\pi y)+c,
$$
where $c$ is a constant with $c>1$. From now on, let's consider $\mathcal{V}=\{V_\alpha\}_{\alpha>0}$ as a class of symmetric potentials $V_\alpha:\mathbb{R}\to\mathbb{R}$ that satisfy the following conditions:
\begin{itemize}
	\item[($V_1$)] $V_{\alpha}(t)\geq 0$ for all $t\in\mathbb{R}$ and $V_{\alpha}(t)=0\Leftrightarrow t=-\alpha,\alpha$ for $\alpha>0$.
	\item[($V_{2}$)] $V_{\alpha}(-t)=V_{\alpha}(t)$ for any $t\in\mathbb{R}$ and $V^{''}_{\alpha}(\pm\alpha)>0$.
\end{itemize}
With respect to class $\mathcal{V}$, we will also assume the following local uniform estimate involving the root $\alpha$ of $V_\alpha$:
\begin{itemize}
	\item[($V_3$)] Given $\lambda>0$, there exists $C=C(\lambda)>0$ independent of $\alpha\in(0,\lambda)$ such that
	$$
	\displaystyle\max_{|t|\in[0,\alpha]}|V'_\alpha(t)|\leq C, \quad \forall V_\alpha\in\mathcal{V} \quad \text{with} \quad \alpha\in(0,\lambda).
	$$
\end{itemize}
An important example of a family $\mathcal{V}$ of potentials $V_\alpha$, for which conditions $(V_1)$-$(V_3)$ are satisfied is given by: 
$$
V_\alpha(t)=(t^2-\alpha^2)^2,~~\alpha>0.
$$

Recently, Kurganov and Rosenau in \cite{Kurganov} reported an interesting physical motivation for investigating transition-type solutions to Allen-Cahn type equations involving the prescribed mean curvature, as exemplified by equation \eqref{E1}. Motivated by \cite{Kurganov}, Bonheure, Obersnel and Omari in \cite{Bonheure1} were led to minimization problems defined in certain function classes in $BV_{\text{loc}}(\mathbb{R})$ to investigate the existence of a basic heteroclinic solution, which naturally connect the stationary points $\pm1$, of the one-dimensional equation
\begin{equation}\label{E2}
	-\left(\frac{q'}{\sqrt{1+(q')^2}}\right)'+a(t)V'(q)=0~~\text{in}~~\mathbb{R},
\end{equation}
where $V$ a double-well potential with minima at $t=\pm1$ and $a$ is asymptotic at infinity to a positive periodic function with $0<\displaystyle\operatorname{ess}\inf_{t\in\mathbb{R}}a(t)$. After, Alves and Isneri in \cite{Alves1} also explored the existence of basic heteroclinic solution for \eqref{E2} in the case where $a\in L^\infty(\mathbb{R})$ is an even non-negative function with $0<\displaystyle\inf_{t\geq M}a(t)$ for some $M>0$. In that paper, inspired to the one introduced by Obersnel and Omari in \cite{Omari}, the authors used cutting techniques on the mean curvature operator to build up a variational framework on the usual Orlicz-Sobolev space $W^{1,\Phi}_{\text{loc}}(\mathbb{R})$ to show that \eqref{E2} has a  solution $q$ such that $$\lim_{t\to\pm\infty}q(t)=\pm\alpha$$ whenever the distance between the roots $\pm\alpha$ of the symmetric potential $V$ is small. When $a(t)=1$, it was shown in the recent work by Alves, Isneri and Montecchiari \cite{Alves2} that under certain conditions in potential $V$, there exists $\alpha_0>0$ such that for each $\alpha\in(0,\alpha_0)$ the equation \eqref{E2} has, up to translation, a unique twice differentiable heteroclinic solution $q$ in $C^{1,\gamma}_{\text{loc}}(\mathbb{R})$ for some $\gamma\in(0,1)$ satisfying the following exponential decay estimates 
\begin{equation*}\label{exp1}
	0<\alpha-q(t)\leq\theta_1e^{-\theta_2t}~~\text{ and }~~0<q'(t)\leq \beta_1e^{-\beta_2t}~~\text{ for all }t\geq0
\end{equation*}
and 
\begin{equation*}\label{exp2}
	0<\alpha+q(t)\leq \theta_3e^{\theta_4t}~~\text{ and }~~0<q'(t)\leq \beta_3e^{\beta_4t}~~\text{ for all }t\leq0,
\end{equation*}
for some real numbers $\theta_i,\beta_i>0$. When $V(t)=\left(t^2-\alpha^2\right)^2$ the exponential estimates above are refined to 
\begin{equation*}\label{th0}
	\alpha\tanh\left(\alpha\sqrt{2} t\right)\leq q(t)\leq \alpha\tanh\left(\frac{\alpha\sqrt{2}t}{\kappa}\right)~\text{ for }~t\geq 0,
\end{equation*}
where $\kappa$ is a positive constant that depends on $\|q'\|_{L^{\infty}(\mathbb{R})}$. Heteroclinic-type solutions are also observed in planar scenarios. For instance, in \cite{Alves3}, Alves and Isneri show the existence of such solutions for \eqref{E1}. These solutions, denoted by $u(x,y)$, exhibit periodicity in $y$ with a particular asymptotic behavior:
$$
u(x,y)\rightarrow\alpha\text{ as }x\rightarrow-\infty\text{ and }u(x,y)\rightarrow\beta\text{ as }x\rightarrow+\infty,\text{ uniformly in }y\in\mathbb{R},
$$
where $\alpha$ and $\beta$ represent the minima of a nonsymmetric double-well potential $V$. This phenomenon occurs under the condition that $A(x,y)$ displays oscillatory behavior with respect to the variable $y$ and satisfies some geometric constraints on $x$, such as periodicity, asymptotic periodicity at infinity, and asymptotically away from zero at infinity.

In this study, we show that equation \eqref{E1} exhibits additional transition-type solutions, including the saddle solution, which are more intricate in nature. The key theorems of this paper are outlined below.

\begin{theorem}\label{T1}
	Assume ($A_1$)-($A_4$). Given $L>0$, there exists $\delta>0$ such that for each $V_\alpha\in\mathcal{V}\cap C^2(\mathbb{R},\mathbb{R})$ with $\alpha\in(0,\delta)$, the prescribed mean curvature equation
	\begin{equation*}
		-div\left(\frac{\nabla u}{\sqrt{1+|\nabla u|^2}}\right) + A(x,y)V_{\alpha}'(u)=0~~\text{ in }~~\mathbb{R}^2,
	\end{equation*} 
	possesses a weak solution $v_{\alpha,L}$ in $C^{1,\gamma}_{\text{loc}}(\mathbb{R}^2)$, for some $\gamma\in(0,1)$, satisfying the following properties:
	\begin{itemize}
		\item [(a)] $0<v_{\alpha,L}(x,y)<\alpha$ on the fist quadrant in $\mathbb{R}^2$,
		\item[(b)] $v_{\alpha,L}(x,y)=-v_{\alpha,L}(-x,y)=-v_{\alpha,L}(x,-y)$ for all $(x,y)\in\mathbb{R}^2$,
		\item[(c)] $v_{\alpha,L}(x,y)=v_{\alpha,L}(y,x)$ for any $(x,y)\in\mathbb{R}^2$,
		\item[(d)] $v_{\alpha,L}(x,y)\rightarrow \alpha$ as $x\rightarrow\pm\infty$ and $y\rightarrow\pm\infty$,
		\item[(e)]  $v_{\alpha,L}(x,y)\rightarrow -\alpha$ as $x\rightarrow\mp\infty$ and $y\rightarrow\pm\infty$,
		\item[(f)] $\|\nabla v_{\alpha,L}\|_{L^{\infty}(\mathbb{R}^2)}\leq \sqrt{L}$.
	\end{itemize}
\end{theorem}

The items $(d)$ and $(e)$ of the theorem above tells us that along directions parallel to the axes, the saddle solution $v_{\alpha,L}$ is uniformly asymptotic to stationary solutions $\pm\alpha$. Moreover, when $A(x,y)$ is a positive constant, we obtain an infinite number of geometrically distinct saddle-type solutions to the equation \eqref{E1}. This observation is detailed in the ensuing result. 

\begin{theorem}\label{T2}
	Assume that $A(x,y)$ is a positive constant $b$. Then, given $L>0$, there exists $\delta>0$ such that for each $V_\alpha\in\mathcal{V}\cap C^2(\mathbb{R},\mathbb{R})$ with $\alpha\in(0,\delta)$ and for each $j\geq2$, the prescribed mean curvature equation
	\begin{equation*}
		-div\left(\frac{\nabla u}{\sqrt{1+|\nabla u|^2}}\right) + bV_{\alpha}'(u)=0~~\text{ in }~~\mathbb{R}^2,
	\end{equation*}
 	possesses a weak solution $v_{\alpha,L,j}$ in $C^{1,\gamma}_{\text{loc}}(\mathbb{R}^2)$, for some $\gamma\in(0,1)$, such that if \linebreak $\tilde{v}_{\alpha,L,j}(\rho,\theta)=v_{\alpha,L,j}(\rho\cos(\theta),\rho\sin(\theta))$, then
	\begin{itemize}
		\item [(a)] $0<\tilde{v}_{\alpha,L,j}(\rho,\theta)<\alpha$ for any $\theta\in[\frac{\pi}{2}-\frac{\pi}{2j},\frac{\pi}{2})$ and $\rho>0$,
		\item[(b)] $\tilde{v}_{\alpha,L,j}(\rho,\frac{\pi}{2}+\theta)=-\tilde{v}_{\alpha,L,j}(\rho,\frac{\pi}{2}-\theta)$ for all $(\rho,\theta)\in[0,+\infty)\times\mathbb{R}$,
		\item[(c)] $\tilde{v}_{\alpha,L,j}(\rho,\theta+\frac{\pi}{j})=-\tilde{v}_{\alpha,L,j}(\rho,\theta)$ for all $(\rho,\theta)\in[0,+\infty)\times\mathbb{R}$,
		\item[(d)] $\tilde{v}_{\alpha,L,j}(\rho,\theta)\rightarrow(-\alpha)^{k+1}$ as $\rho\rightarrow+\infty$ whenever $\theta\in\left(\frac{\pi}{2}+k\frac{\pi}{j},\frac{\pi}{2}+(k+1)\frac{\pi}{j}\right)$ for $k=0,\ldots,2j-1$,
		\item[(e)] $\|\nabla v_{\alpha,L,j}\|_{L^{\infty}(\mathbb{R}^2)}\leq \sqrt{L}$.
	\end{itemize}
\end{theorem}

In other words, the saddle-type solution $\tilde{v}_{\alpha,L,j}$ is antisymmetric with respect to the half-line $\theta=\frac{\pi}{2}$, $\frac{2\pi}{j}$-periodic in the angle variable and has $L^\infty$-norm less than or equal to $\alpha$ with the asymptotic behavior at infinity described in item $(d)$. Moreover, the solutions $v_{\alpha,L,j}$ described in the theorem above are characterized by the fact that, along different directions parallel to the end lines, they are uniformly asymptotic to $\pm\alpha$ and such solutions may appropriately be termed ``pizza solutions''. 

In order to establish Theorems \ref{T1} and \ref{T2}, it was necessary improve the findings presented in the works of Alves and Isneri \cite{Alves4,Alves5}, particularly concerning saddle-type solutions, for a larger class of $N$-functions. For a comprehensive understanding of these advancements, readers are encouraged to refer directly to Section \ref{s2} of this work. Moreover, as far as we know, Theorems \ref{T1} and \ref{T2} are the first results in the literature on saddle-type solutions for some stationary Allen-Cahn-type equations involving the prescribed mean curvature operator in the whole plane. Transition-type solutions to equations involving the prescribed mean curvature operator is an extremely fascinating field of mathematics and there are still many open questions one can work on. For example, a possible extension to Theorems \ref{T1} and \ref{T2} would be to study the existence of a saddle solution for \eqref{E1} without requiring that the distance between the absolute minima $\pm\alpha$ of $V$ be small. We believe that a natural approach to solve such a problem would be to look for minima of an action functional on a convex subset of the space of functions of bounded variation $BV_{\text{loc}}(\mathbb{R})$. 

The structure of this manuscript unfolds as follows: In Section \ref{s2}, we delve into the analysis of the existence of saddle-type solutions for a specific class of quasilinear elliptic equations that complements and extends the investigations presented in \cite{Alves4,Alves5}. Taking into account the study in Section \ref{s2}, in Section \ref{s3} we will present the proof of our main results. Lastly, in Section \ref{s4}, we will make some final comments about our study.


\section{Existence of saddle solutions for quasilinear  equations}\label{s2}

We will show in this section that the main results about saddle-type solutions in works \cite{Alves4,Alves5} are extended to a larger class of $N$-functions than the class that was presented there. We will start with a brief review of what has been done in \cite{Alves5} for saddle solutions. To recapitulate, in that paper, the authors studied the existence of transition-type solutions for the following class of quasilinear elliptic equations given by
\begin{equation}\label{Equation2.1}
	-\Delta_{\Phi}u + A(x,y)V'(u)=0~~\text{ in }~~\mathbb{R}^2,
\end{equation}
where the conditions $(A_1)$-$(A_4)$ were assumed on $A$ and on $V$, $(V_{1})$-$(V_{2})$ and
\begin{itemize}
	\item[($V_{4}$)] There are $\delta_\alpha\in(0,\alpha)$ and $w_1, w_2>0$ such that 
	$$
	w_1\Phi(|t-\alpha|)\leq V(t)\leq w_2\Phi(|t-\alpha|)~~\forall t\in(\alpha-\delta_\alpha,\alpha+\delta_\alpha).
	$$
	\item[($V_{5}$)] There are $\omega_1,\omega_2,\omega_3,\omega_4,\tau>0$ such that 
	$$
	-\omega_3\phi(\omega_4|\alpha-t|)(\alpha-t)t\leq V'(t)\leq-\omega_1\phi(\omega_2|\alpha-t|)(\alpha-t)t~~\forall t\in[0,\alpha+\tau].
	$$ 
\end{itemize}
With respect to $\Phi:\mathbb{R}\to[0,+\infty)$, $\Phi$ is designated as an $N$-function, that is, it satisfies precisely the following properties
\begin{itemize}
	\item[(a)] $\Phi$ is continuous, convex and even,
	\item[(b)] $\Phi(t)=0$ if and only if $t=0$,
	\item[(c)] $\displaystyle \lim_{t\rightarrow 0}\dfrac{\Phi(t)}{t}=0$ and $\displaystyle \lim_{t\rightarrow +\infty}\dfrac{\Phi(t)}{t}=+\infty$.
\end{itemize}
Specifically, $\Phi$ has the form
$$
\Phi(t)=\int_0^{|t|}\phi(s)sds,
$$
with $\phi\in C^{1}([0,+\infty), [0,+\infty))$ satisfying
\begin{itemize}
	\item[($\phi_1$)] $\phi(t)>0$ and $(\phi(t)t)'>0$ for any $t>0$.
	\item[($\phi_2$)] There are $l,m\in\mathbb{R}$ with $1<l\leq m$ such that 
	$$
	l-1\leq \dfrac{(\phi(t)t)'}{\phi(t)}\leq m-1~~\text{for all}~~ t>0.
	$$
	\item[($\phi_3$)] There exist constants $c_1,c_2,\eta>0$ and $s>1$ satisfying
	$$
	c_1t^{s-1}\leq \phi(t)t\leq c_2t^{s-1}~~\text{for}~~t\in[0,\eta].
	$$	
	\item[($\phi_4$)] $\phi$ is non-decreasing on $(0,+\infty)$.
\end{itemize}
Finally, $\Delta_{\Phi}$  emerges as a quasilinear operator in divergence form, succinctly expressed as:
$$
\Delta_{\Phi}u=\text{div}(\phi(|\nabla u|)\nabla u).
$$

The conditions $(\phi_1)$-$(\phi_3)$ establish the groundwork for ensuring the existence of a weak solution $u_0$ to \eqref{Equation2.1} with specific properties. Firstly, $u_0$ is $1$-periodic in the variable $y$, odd symmetry with respect to $x$, and furthermore, it must satisfy
$$
0<u_0(x,y)<\alpha~~\text{for all}~~x>0. 
$$
Moreover, $u_0$ is a heteroclinic type solution, that is,
$$
u_0(x,y)\rightarrow-\alpha\text{ as }x\rightarrow-\infty\text{ and }u_0(x,y)\rightarrow\alpha\text{ as }x\rightarrow+\infty\text{ uniformly in }y\in\mathbb{R}.
$$
The introduction of condition $(\phi_4)$ serves a crucial purpose, allowing for the utilization of exponential decay estimates specifically for $u_0\pm\alpha$. This condition facilitates the derivation of an inequality crucial for further analysis, presented as 
\begin{equation}\label{5E1}
	\phi(|\zeta'(x)|)\leq \phi(\omega_2\zeta(x))~~\text{for all}~~x\in\mathbb{R},
\end{equation}
which involves the real function $\zeta:\mathbb{R}\to\mathbb{R}$ given by
\begin{equation}\label{5E2}
	\zeta(x)=\delta_\alpha\frac{\cosh\left(a\left(x-\frac{j-\frac{1}{4}+L}{2}\right)\right)}{\cosh\left(a\frac{j-\frac{1}{4}-L}{2}\right)},
\end{equation}
where $\omega_2$, $L$ and $j$ are chosen properly and the constant $a$ is small enough. The role of inequality \eqref{5E1} is pivotal as it permits direct calculations leading to the desired exponential-type estimates. For a comprehensive understanding of the intricate details and rigorous verification, readers are directed to Lemmas 4.4, 4.5, and 4.6 in the referenced work \cite{Alves5}. This study of the asymptotic behavior at infinity of heteroclinic solutions plays a fundamental role in the search of saddle-type solutions to \eqref{Equation2.1}, as outlined in the methodology employed in \cite{Alves4,Alves5}.

Our goal now is to replace $(\phi_4)$ with another condition on $\phi$ that includes a larger number of examples for $\phi$, including the monotonically non-decreasing functions. Specifically, we introduce the following assumption
\begin{itemize}
	\item[$(\tilde{\phi}_4)$] There exist positive constants $\kappa_1$ and $\kappa_2$ such that the inequality $$\phi(|\zeta'(t)|)\leq \kappa_1\phi(\kappa_2\zeta(t)) \quad \text{for all}\quad t\in\mathbb{R}$$ holds whenever $a>0$ is sufficiently small in \eqref{5E2}. 
\end{itemize}
The reader is invited to conduct a meticulous verification of the exponential decay estimates relative to the heteroclinic solutions as delineated in references \cite{Alves4,Alves5}, specifically observing their occurrence when substituting $(\phi_4)$ by $(\tilde{\phi}_4)$. It should be noted that functions $\phi$ satisfying $(\phi_4)$ also satisfy $(\tilde{\phi}_4)$. However, it's important to observe that the converse is not true, as certain functions, compliant with $(\tilde{\phi}_4)$, may not necessarily satisfy $(\phi_4)$, thereby establishing $(\tilde{\phi}_4)$ as a broader framework compared to $(\phi_4)$. In the following section, we provide an explicit example elucidating this distinction.

Under our current assumptions on $\phi$, we provide below a result that generalizes Theorem 1.2 of \cite{Alves5}.

\begin{theorem}\label{TE1}
	Assume ($\phi_1$)-($\phi_3$), ($\tilde{\phi}_4$), $V\in C^1(\mathbb{R},\mathbb{R})$, ($V_{1}$)-($V_{2}$), ($V_{4}$)-($V_{5}$) and ($A_1$)-($A_4$). Then, there is $v\in C^{1,\gamma}_{\text{loc}}(\mathbb{R}^2)$ for some $\gamma\in(0,1)$ such that $v$ is a weak solution of \eqref{Equation2.1} that verifies the following: 
	\begin{itemize}
		\item [(a)] $0<v(x,y)<\alpha$ on the fist quadrant in $\mathbb{R}^2$,
		\item[(b)] $v(x,y)=-v(-x,y)=-v(x,-y)$ for all $(x,y)\in\mathbb{R}^2$,
		\item[(c)] $v(x,y)=v(y,x)$ for any $(x,y)\in\mathbb{R}^2$,
		\item[(d)] $v(x,y)\rightarrow \alpha$ as $x\rightarrow\pm\infty$ and $y\rightarrow\pm\infty$,
		\item[(e)]  $v(x,y)\rightarrow -\alpha$ as $x\rightarrow\mp\infty$ and $y\rightarrow\pm\infty$.
	\end{itemize}
\end{theorem}

When $A(x,y)$ is a positive constant, additional conditions regarding the geometric properties of the potential $V$ were deemed necessary within the framework presented in \cite{Alves4} to establish the existence of infinitely many geometrically distinct saddle-type solutions of equation \eqref{Equation2.1}.  Specifically, the following conditions on $V$ were explored ($V_{1}$)-($V_{2}$), ($V_{4}$)-($V_{5}$) and 
\begin{itemize}
	\item[$(V_{6})$] There exists $\delta_0>0$ such that $V'$ is increasing on $(\alpha-\delta_0,\alpha)$.  
	\item[$(V_{7})$] There exist $\gamma,\epsilon>0$ such that 
	$$
	\tilde{\Phi}(V'(t))\leq\gamma\Phi(|\alpha-t|),~~\forall t\in(\alpha-\epsilon,\alpha).
	$$
\end{itemize}
Therefore, motivated by Theorem \ref{TE1}, we can follow the steps of work \cite{Alves4} to derive the subsequent result.

\begin{theorem}\label{TE2}
	Assume ($\phi_1$)-($\phi_3$), ($\tilde{\phi}_4$), ($V_{1}$)-($V_{2}$), ($V_{4}$)-($V_{7}$) and that $A(x,y)$ is a positive constant. Then, For each $j\geq2$, there exists $v_j\in C^{1,\gamma}_{\text{loc}}(\mathbb{R}^2)$ for some $\gamma\in(0,1)$ such that $v_j$ is a weak solution of \eqref{Equation2.1}. Moreover, by setting $\tilde{v}_j(\rho,\theta)=v_j(\rho\cos(\theta),\rho\sin(\theta))$, we have that
	\begin{itemize}
		\item [(a)] $0<\tilde{v}_j(\rho,\theta)<\alpha$ for any $\theta\in[\frac{\pi}{2}-\frac{\pi}{2j},\frac{\pi}{2})$ and $\rho>0$,
		\item[(b)] $\tilde{v}_j(\rho,\frac{\pi}{2}+\theta)=-\tilde{v}_j(\rho,\frac{\pi}{2}-\theta)$ for all $(\rho,\theta)\in[0,+\infty)\times\mathbb{R}$,
		\item[(c)] $\tilde{v}_j(\rho,\theta+\frac{\pi}{j})=-\tilde{v}_j(\rho,\theta)$ for all $(\rho,\theta)\in[0,+\infty)\times\mathbb{R}$,
		\item[(d)] $\tilde{v}_j(\rho,\theta)\rightarrow(-\alpha)^{k+1}$ as $\rho\rightarrow+\infty$ whenever $\theta\in\left(\frac{\pi}{2}+k\frac{\pi}{j},\frac{\pi}{2}+(k+1)\frac{\pi}{j}\right)$ for \linebreak $k=0,\ldots,2j-1$.
	\end{itemize}
\end{theorem}

In fact, the conditions on $V$ in the aforementioned theorem can be refined; namely, conditions $(V_{6})$ and $(V_{7})$ can be omitted. To illustrate this, one can combine several arguments from the paper \cite{Alves5} and implement the concept of partitioning $\mathbb{R}^2$ into $2j$ disjoint triangular sets, where $j\geq2$, as described in the construction of saddle solutions presented in \cite[Section 6]{Alves4}. As a result, we can present the following theorem, with the details of its verification left to the reader.

\begin{theorem}\label{TE3}
	Assume ($\phi_1$)-($\phi_3$), ($\tilde{\phi}_4$), ($V_{1}$)-($V_{2}$), ($V_{4}$)-($V_{5}$) and that $A(x,y)$ is a positive constant. Then, For each $j\geq2$, there exists $v_j\in C^{1,\gamma}_{\text{loc}}(\mathbb{R}^2)$ for some $\gamma\in(0,1)$ such that $v_j$ is a weak solution of \eqref{Equation2.1}. Moreover, by setting $\tilde{v}_j(\rho,\theta)=v_j(\rho\cos(\theta),\rho\sin(\theta))$, we have that
	\begin{itemize}
		\item [(a)] $0<\tilde{v}_j(\rho,\theta)<\alpha$ for any $\theta\in[\frac{\pi}{2}-\frac{\pi}{2j},\frac{\pi}{2})$ and $\rho>0$,
		\item[(b)] $\tilde{v}_j(\rho,\frac{\pi}{2}+\theta)=-\tilde{v}_j(\rho,\frac{\pi}{2}-\theta)$ for all $(\rho,\theta)\in[0,+\infty)\times\mathbb{R}$,
		\item[(c)] $\tilde{v}_j(\rho,\theta+\frac{\pi}{j})=-\tilde{v}_j(\rho,\theta)$ for all $(\rho,\theta)\in[0,+\infty)\times\mathbb{R}$,
		\item[(d)] $\tilde{v}_j(\rho,\theta)\rightarrow(-\alpha)^{k+1}$ as $\rho\rightarrow+\infty$ whenever $\theta\in\left(\frac{\pi}{2}+k\frac{\pi}{j},\frac{\pi}{2}+(k+1)\frac{\pi}{j}\right)$ for \linebreak $k=0,\ldots,2j-1$.
	\end{itemize}
\end{theorem}


\section{Saddle solution of the prescribed mean curvature equation}\label{s3}

Our primary goal in this section is to establish the existence of a saddle solution to the prescribed mean curvature equation \eqref{E1} under the condition that the global minima of the potential $V$ are sufficiently close. For this reason, we will index the number $\alpha$ in $V$, denoting it as $V_\alpha$, and throughout this section, we will assume that $V_\alpha\in\mathcal{V}\cap C^2(\mathbb{R},\mathbb{R})$. To achieve this, we first study an auxiliary problem of the form \eqref{Equation2.1}, proving the existence of a saddle solution within this context. The idea here is similar to those presented in work \cite{Alves1}, so the exposition will be brief. For our purposes, we employ a truncation scheme for the prescribed mean curvature operator, defined for each $L>0$ as follows:
$$
\varphi_L(t)=\left\{\begin{array}{lll}
	\dfrac{1}{\sqrt{1+t}},&\mbox{if}\quad t\in[0,L],
	\\
	x_L(t-L-1)^2+y_L,&\mbox{if}\quad t\in[L,L+1],
	\\
	y_L,&\mbox{if}\quad t\in[L+1,+\infty),
\end{array}\right.
$$
where the numbers $x_L$ and $y_L$ are expressed by
$$
x_L=\frac{\sqrt{1+L}}{4(1+L)^2}~\text{ and }~y_L=(4L+3)x_L.
$$
As a consequence, for each fixed $L>0$, we get the following quasilinear equation
\begin{equation}\label{E3}
	-\Delta_{\Phi_L}u+A(x,y)V'_\alpha(u)=0~~\text{in}~~\mathbb{R}^2,
\end{equation}
where $\Phi_L:\mathbb{R}\rightarrow[0,+\infty)$ is an $N$-function of the form
$$
\Phi_L(t)=\int_{0}^{|t|}\phi_L(s)sds~~\text{with}~~\phi_L(t)=\varphi_L(t^2).
$$

The following two lemmas will be frequently referenced throughout this work and they can be found in \cite{Alves1,Alves2}.

\begin{lemma}\label{L1}
	For each $L>0$, the functions $\phi_L$ and $\Phi_L$ satisfy the following properties:
	\begin{itemize}
		\item [(a)] $\phi_L$ is $C^1$.
		\item [(b)] $y_L\leq \phi_L(t)\leq 1$ for each $t\geq 0$.
		\item [(c)] $\dfrac{y_Lt^2}{2}\leq \Phi_L(t)\leq \dfrac{t^2}{2}$ for any $t\geq0$.
		\item [(d)] $\Phi_L$ is a convex function. 
		\item [(e)] $\left(\phi_L\left(t\right)t\right)'>0$ for all $t>0$.
	\end{itemize}
\end{lemma}

\begin{lemma}\label{L2}
		For each $L>0$, it turns out that $\phi_L$ satisfies $(\phi_1)$-$(\phi_3)$.
\end{lemma}

A direct check shows that $\phi_L$ exhibits a monotonically non-increasing behavior over the interval $(0,+\infty)$. Consequently, $\phi_L$ does not satisfy $(\phi_4)$. However, each $\phi_L$ satisfies $(\tilde{\phi}_4)$, as the following lemma says.

\begin{lemma}\label{L}
	For each $L>0$, the function $\phi_L$ satisfies $(\tilde{\phi}_4)$.
\end{lemma}
\begin{proof}
	Indeed, according to Lemma \ref{L1}-$(b)$ we have that
	$$
	y_L\leq \phi_L(t)\leq 1~~\text{for all}~~t\geq 0.
	$$
	In particular, 
	$$
	y_L\leq \phi_L(|\zeta'(t)|), \phi_L(\zeta(t))\leq 1~~\text{for all}~~ t\in\mathbb{R},
	$$	
	and therefore, 
	$$
	\phi_L(|\zeta'(t)|)\leq \frac{1}{y_L}\phi_L(\zeta(t))~~\text{for all}~~ t\in\mathbb{R}. 
	$$
	Finally, since there is no restriction for the constant $a>0$ in \eqref{5E2}, the lemma follows.
\end{proof}

The evidence provided in Lemma \ref{L} leads to the following observation.

\begin{remark}
	By the argument of the previous lemma, we can conclude that any function $\phi:[0,+\infty)\to[0,+\infty)$ satisfying 
	$$
	c_1\leq \phi(t)\leq c_2~~\text{for all}~~t\geq 0
	$$ 
	for some constants $c_1,c_2>0$ satisfies condition $(\tilde{\phi}_4)$.
\end{remark}

To find saddle solutions for the equation \eqref{E1}, our approach involves examining the existence of these solutions within the framework of the auxiliary problem \eqref{E3}. To begin with, note that the condition $V_\alpha\in \mathcal{V}\cap C^2(\mathbb{R},\mathbb{R})$ implies that $V_\alpha$ satisfies the conditions $(V_{4})$ with respect to $\Phi_L$ and $(V_{5})$ concerning $\phi_L$. Indeed, by $(V_{2})$ it is possible to find numbers $\rho_1,\rho_2,d_1,d_2,d_3>0$ such that 
$$
|V_\alpha'(t)|\leq d_1|t-\alpha|~~\text{for all}~~t\in[\alpha-\rho_1,\alpha+\rho_1]
$$
and
$$
d_2|t-\alpha|^2\leq V_\alpha(t)\leq d_3|t-\alpha|^2~~\text{for all}~~ t\in(\alpha-\rho_2,\alpha+\rho_2),
$$
from which it follows by Lemma \ref{L1} that
$$
|V_\alpha'(t)|\leq \frac{d_1}{y_L}\phi_L(|t-\alpha|)|t-\alpha|~~\text{for all}~~ t\in[\alpha-\rho_1,\alpha+\rho_1]
$$
and
$$
2d_2\Phi_L(|t-\alpha|)\leq V_\alpha(t)\leq \frac{2d_3}{y_L}\Phi_L(|t-\alpha|)~~\text{for all}~~ t\in(\alpha-\rho_2,\alpha+\rho_2).	
$$	
Hence, we may use Theorem \ref{TE1} to obtain, for each $L>0$, a saddle solution denoted as $v_{\alpha,L}:\mathbb{R}^2\to\mathbb{R}$ to equation \eqref{E3}. This solution, $v_{\alpha,L}$, is a weak solution to \eqref{E3} belonging to $C^{1,\gamma}_{\text{loc}}(\mathbb{R}^2)$, for some $\gamma\in(0,1)$, satisfying
$$
|v_{\alpha,L}(x,y)|\leq\alpha~~\text{ for all }~~(x,y)\in\mathbb{R}^2.
$$ 
Moreover, it preserves the same sign pattern as $xy$, odd in both the variables $x$ and $y$, symmetric with respect to the diagonals $y=\pm x$ and presents the asymptotic behavior described in items $(d)$ and $(e)$ of Theorem \ref{TE1}. We are now going to use this information together with condition $(V_3)$ to establish an estimate involving the functions $v_{\alpha,L}$.

\begin{lemma}\label{LE51}
	There exists $\delta>0$ such that for each $V_{\alpha} \in \mathcal{V}$ with $\alpha\in(0,\delta)$, the function $v_{\alpha,L}$ satisfies following the estimate below  
	\begin{equation}\label{E8}
		\|v_{\alpha,L}\|_{C^1(B_{1}(z))}<\sqrt{L},
	\end{equation}
	where $B_{1}(z)$ denotes any open ball in $\mathbb{R}^2$ with radius 1.
\end{lemma}
\begin{proof}
	To establish this lemma, we employ a proof by contradiction. Let us assume the contrary. Then for each $n \in \mathbb{N}$ there exists a potential $V_{\alpha_n,\beta_n}\in \mathcal{V}\cap C^2(\mathbb{R},\mathbb{R})$ with $\alpha\in\left(0,\frac{1}{n}\right)$ and $z_n=(z_{n,1},z_{n,2})\in\mathbb{R}^2$ such that the function $v_{\alpha_n,L}$ satisfies
	\begin{equation}\label{E6}
		\|v_{\alpha_n,L}\|_{C^1(B_1(z_n))}\geq \sqrt{L}~~\text{for all}~~ n\in\mathbb{N}.
	\end{equation} 
	For the purpose of our analysis, we will study the regularity of specific solutions to the following elliptic equation: 
	\begin{equation}\label{E4}
		-\Delta_{\Phi_{L}}u+B_n(x,y)=0~~\text{ in }~~\mathbb{R}^2,
	\end{equation}
	where the scalar measurable function $B_n:\mathbb{R}^2\to\mathbb{R}$ is defined as:
	$$
	B_{n}(x,y)=A(x+z_{n,1},y+z_{n,2})V_{\alpha_n}'(v_{\alpha_n,L}(x,y)).
	$$
	Under condition $(V_3)$, $B_{n}$ satisfies the subsequent estimate: 
	\begin{equation}\label{NI}
		|B_{n}(x,y)|\leq C\|A\|_{L^{\infty}(\mathbb{R}^2)}~~\forall(x,y)\in\mathbb{R}^2\text{ and }\forall n\geq n_0,
	\end{equation}
	for some $n_0\in\mathbb{N}$ and $C>0$, where $C$ is independent of $n$. Indeed, it's worth noting that the condition $\alpha_n\to 0$ as $n\to+\infty$ implies the existence of $n_0\in\mathbb{N}$ such that $\alpha_n\in(0,1)$ for all $n\geq n_0$. Consequently, by $(V_3)$, it follows that there exists a constant $C=C(1)>0$ independent of $\alpha\in(0,1)$ satisfying
	$$
	\displaystyle\max_{|t|\in[0,\alpha]}|V'_\alpha(t)|\leq C, \quad \forall V_\alpha\in\mathcal{V} \quad \text{with} \quad \alpha\in(0,1).
	$$
	Now, since $\alpha_n\in (0,1)$ for all $n\geq n_0$ and 
	$$
	|v_{\alpha_n,L}(x,y)|\leq\alpha_n\quad \forall (x,y)\in\mathbb{R}^2\quad \forall n\in\mathbb{N},
	$$ 
	we obtain  
	$$
	|V'_{\alpha_n}(v_{\alpha_n,L}(x,y))|\leq C,~~\forall (x,y)\in\mathbb{R}^2~~\text{and}~~ \forall n\geq n_0,
	$$
	which is sufficient to ensure inequality \eqref{NI}. A particularly relevant weak solution to \eqref{E4} is given by:
	$$
	u_n(x,y)=v_{\alpha_n,L}(x+z_{n,1},y+z_{n,2})~\text{ for }~(x,y)\in\mathbb{R}^2.
	$$
	Now, utilizing inequality \eqref{NI}, we can employ elliptic regularity estimates on $u_n$, as developed by Lieberman in \cite[Theorem 1.7]{Lieberman}. This yields $u_n\in C^{1,\gamma_0}_{\text{loc}}(\mathbb{R}^2)$ for some $\gamma_0\in(0,1)$, and
	$$
	\|u_n\|_{C^{1,\gamma_0}_{\text{loc}}(\mathbb{R}^2)}\leq R~~\text{for all}~~n\geq n_0,
	$$
	where $R$ is a positive constant independent of $n$. Consequently, from the above estimate, employing Arzel\`a-Ascoli's theorem, we deduce the existence of $u\in C^1(B_1(0))$ and a subsequence of $(u_n)$, still denoted by $(u_n)$, such that
	\begin{equation}\label{E5}
		u_n\to u~~\text{ in }~~C^1(B_1(0)).
	\end{equation}
	Since $\|v_{\alpha_n,L}\|_{L^{\infty}(\mathbb{R}^2)}$ goes to $0$ as $\alpha_n\to 0$, 
	$$
	\|u_n\|_{L^{\infty}(\mathbb{R}^2)}\to 0~~\text{as}~~n\to+\infty,
	$$
	thus from the convergence \eqref{E5}, we naturally infer that $u=0$ on $B_1(0)$. Therefore, there exists $\tilde{n}\in\mathbb{N}$ such that:
	$$
	\|u_n\|_{C^1\left(B_1(0)\right)}<\sqrt{L}~~\text{for all}~~n\geq \tilde{n},
	$$
	which, arising from the definition of $u_n$, implies:
	$$
	\|v_{\alpha_n,L}\|_{C^1(B_1(z_n))}<\sqrt{L}~~\text{for all}~~n\geq \tilde{n},
	$$ 
	which contradicts \eqref{E6}. Hence, the proof of the lemma is complete.
\end{proof}

The estimate \eqref{E8} will play a key role in the theorem below, facilitating the investigation of the existence of saddle solutions for the prescribed mean curvature equation \eqref{E1} whenever the roots $\pm\alpha$ of $V$ are sufficiently close.

\noindent {\bf The proof of Theorem \ref{T1}.} 

We claim that for each fixed $L>0$ there exists $\delta>0$ such that for each $V_{\alpha} \in \mathcal{V}\cap C^2(\mathbb{R},\mathbb{R})$ with $\alpha\in(0,\delta)$, the function  $v_{\alpha,L}$ satisfies the estimate below  
	\begin{equation}\label{E7}
		\|\nabla v_{\alpha,L}\|_{L^{\infty}(\mathbb{R}^2)}\leq \sqrt{L}.
	\end{equation}
	In fact, given any $(x,y)\in\mathbb{R}^2$ we can choose a point $z\in\mathbb{R}^2$ such that $(x,y)\in B_1(z)$. By virtue of Lemma \ref{LE51}, there exists $\delta>0$ such that if $V_{\alpha} \in \mathcal{V}\cap C^2(\mathbb{R},\mathbb{R})$ and $\alpha\in(0,\delta)$ $\alpha\in(0,\delta)$ one gets
	$$
	\|v_{\alpha,L}\|_{C^1(B_{1}(z))}<\sqrt{L}.
	$$
	In particular,
	$$
	\|\nabla v_{\alpha,L}\|_{L^{\infty}(B_{1}(z))}\leq \sqrt{L}.
	$$
	Therefore, the claim \eqref{E7} is valid from the arbitrariness of $(x,y)\in\mathbb{R}^2$, and thus, the estimate \eqref{E7} guarantees that $v_{\alpha,L}$ is a saddle solution to the equation \eqref{E1} satisfying items $(a)$ to $(f)$, thanks to the study developed in the auxiliary problem \eqref{E3}. This completes the proof. \hspace{0.5cm} $\square$

In the case where $A(x,y)$ is a positive constant, the existence of infinitely many saddle-type solutions for the prescribed mean curvature equation \eqref{E1} can be established. These solutions, bearing a geometry reminiscent of pizzas, may be named as "pizza solutions". Employing the strategy outlined in proving Theorem \ref{T1}, we can prove the Theorem \ref{T2}, which is a multiplicity result concerning saddle-type solutions to \eqref{E1}.


\section{Final remarks}\label{s4}

In this last section, we aim to highlight that although we have made advancements over the results presented in works \cite{Alves4,Alves5}, particularly concerning saddle solutions to quasilinear elliptic equations represented by
\begin{equation}\label{E10}
	-\Delta_{\Phi}u + A(x,y)V'(u)=0~~\text{ in }~~\mathbb{R}^2
\end{equation}
by imposing condition $(\tilde{\phi}_4)$ on $\phi$ instead of $(\phi_4)$, a significant challenge remains unresolved. Specifically, the case $\phi(t)=t^{p-2}$ with $p\in(1,2)$, associated with the $N$-function
$$
\Phi(t)=\frac{|t|^p}{p}~~\text{for}~~p\in(1,2),
$$ 
presents an open problem because, in this case, $\phi$ does not satisfy $(\tilde{\phi}_4)$. This fact arises from the observation that through a direct calculation, we have
$$
\frac{\phi(|\zeta'(t)|)}{\phi(\kappa_2\zeta(t))}\to +\infty~~\text{as}~~|t|\to \frac{j-\frac{1}{4}+L}{2},
$$
for any positive constant $\kappa_2$. This divergence is evident from the behavior of
$$
\left| \frac{\sinh(t)}{\cosh(t)}\right| =\left| \frac{e^{2t}-1}{e^{2t}+1} \right|,
$$
which tends to zero as $|t|\to 0$, resulting in
$$
\left| \frac{\sinh(t)}{\cosh(t)}\right|^{p-2}\to +\infty ~~\text{as}~~|t|\to 0~~\text{for all}~~p\in(1,2).
$$
Finally, it's worth noting that exploring exponential decay type estimates to determine saddle solutions in cases where \eqref{E10} involves the $p$-Laplacian operator with $1<p<2$ is potentially difficult and interesting.

\end{document}